\documentclass[10pt,a4paper]{amsart}
\usepackage{amssymb,mathtools}
\usepackage[margin=1in]{geometry}
\usepackage{extarrows}
\usepackage[all]{xy}
\usepackage{booktabs,tabularx,array,enumitem,microtype}
\usepackage[hidelinks]{hyperref}
\hypersetup{pdftitle={A note on inverse limits of injective modules},
  pdfauthor={Chencheng Zhang},
  pdfkeywords={inverse limit, injective module, cotorsion module, Noetherian ring, hereditary ring}}
\setlist{nosep,leftmargin=2.2em}
\allowdisplaybreaks[2]
\theoremstyle{plain}
\newtheorem*{theorema}{Theorem A}
\newtheorem*{theoremb}{Theorem B}
\newtheorem*{theoremc}{Theorem C}
\newtheorem{lemma}{Lemma}[section]
\newtheorem{proposition}[lemma]{Proposition}

\theoremstyle{remark}
\newtheorem{remark}[lemma]{Remark}

\begin{document}

\title[Inverse limits of injective modules]{A note on inverse limits of injective modules}
\author{Chencheng Zhang}
\thanks{The author was supported by the National Natural Science Foundation of China (No.~12131015).}
\address{School of Mathematical Sciences, Shanghai Jiao Tong University, Shanghai 200240, P. R. China}
\email{zhangchencheng@sjtu.edu.cn}
\date{}

\begin{abstract}
  Let $R$ be a ring and $M$ a left $R$-module.
  We prove that $M$ is an $\omega_1$-indexed inverse limit of injective modules with split epimorphic connecting maps, 
  and $M$ is also an $\omega$-indexed inverse limit of a stationary tower defined by a split monic endomorphism of one injective module.
  These constructions answer two questions of Bergman.
  We then study the class of direct summands of $\omega$-indexed inverse limit of injective modules with epimorphic connecting maps.
  Over a commutative Noetherian ring $R$, they are precisely the Enochs cotorsion modules if and only if $\operatorname{gldim}R\leq1$ and $\operatorname{Spec}R$ is countable.
\end{abstract}

\maketitle

\section{Introduction}\label{sec:introduction}

Throught, all rings are unital and associative, but not necessary commutative.
All modules are unital left modules.
For an index set $I$, write $M^{(I)}$ for the direct sum of $|I|$ copies of $M$, and $M^I$ for the direct product.
Let $\omega$ and $\omega_1$ denote the first infinite ordinal, and the first uncountable ordinal, respectively.

\medskip

Bergman shows that every module is an inverse limit of injective modules with monomorphic connecting maps; moreover, for left Noetherian rings, every module is an inverse limit of injective modules with epimorphic connecting maps \cite[Theorems~2 and~4]{Bergman}.
His Questions~5 ask whether such epimorphic representation can be obtained without the Noetherian hypothesis, even over $\omega_1$; Questions~6 asks whether such monomorphic representation can then be obtained without the Noetherian hypothesis, even over $\omega$.

\medskip

The questions admit affirmative answers, as shown in Theorems~A and~B below.

\begin{theorema}
Let $R$ be a ring and $M$ a left $R$-module.
Then there is an inverse system $(I_\alpha,p_{\alpha\beta})_{\alpha\leq\beta<\omega_1}$ of injective left $R$-modules such that every $p_{\alpha\beta}$ is a split epimorphism and $M\cong\varprojlim_{\alpha<\omega_1} I_\alpha$.
\end{theorema}

\begin{theoremb}
Let $R$ be a ring and $M$ a left $R$-module.
Then there are an injective left $R$-module $I$ and a split monomorphism $\sigma:I\longrightarrow I$ such that $M\cong\varprojlim\bigl(I\xlongleftarrow{\sigma}I \xlongleftarrow{\sigma}I\xlongleftarrow{\sigma}\cdots\bigr) \cong\bigcap_{n<\omega}\sigma^n(I)$.
\end{theoremb}

Bergman's Question~7 concerns countable systems with epimorphic maps.
For this question, write
\begin{align*}
  \mathcal{L}_R&=\left\{\varprojlim_{n<\omega}(I_n,p_n) \mid 
       I_n\text{ is an injective module},\quad p_n:I_{n+1}\longrightarrow I_n
       \text{ is surjective}\right\},\\
  \mathcal{C}_R&=\{C:C\text{ is a direct summand of some }L\in\mathcal{L}_R\}.
\end{align*}
These classes are understood up to isomorphism.
Throughout, \emph{cotorsion} means Enochs cotorsion:
\[
  \operatorname{Cot}(R)=\{C:\operatorname{Ext}^1_R(F,C)=0\text{ for every flat left }R\text{-module }F\}.
\]

\begin{theoremc}
Let $R$ be a nonzero commutative Noetherian ring.
Then
\[
  \mathcal{C}_R=\operatorname{Cot}(R)
  \quad\Longleftrightarrow\quad
  \operatorname{gldim}R\leq1
  \ \text{and}\ |\operatorname{Spec}R|\leq\aleph_0.
\]
In this case, every cotorsion $R$-module $C$ satisfies $D\oplus C\in\mathcal{L}_R$ for some injective $R$-module $D$.
\end{theoremc}

In particular, let $R$ be a Dedekind domain with fraction field $K$.
Then $\mathcal{C}_R=\operatorname{Cot}(R)$ if and only if $K$ is at most countably generated as an $R$-module.
The equality holds for every discrete valuation ring.

\medskip

For technical reasons, Theorem~C concerns the direct-summand closure $\mathcal{C}_R$, rather than the original class $\mathcal{L}_R$.
Indeed, for every prime $p$, the group $\mathbb{Z}/p\mathbb{Z}$ belongs to $\mathcal{C}_{\mathbb{Z}}\setminus\mathcal{L}_{\mathbb{Z}}$.

\section{Split presentations over arbitrary rings}\label{sec:presentations}

For an abelian group $A$, put $\mathsf{C}(A)=\operatorname{Hom}_{\mathbb{Z}}(R,A)$ the coinduced left $R$-module.
Then adjunction gives a natural isomorphism $\operatorname{Hom}_R(N,\mathsf{C}(A))\cong\operatorname{Hom}_{\mathbb{Z}}(N,A)$ for each $R$-module $N$.

\begin{lemma}
  If $A$ is divisible, then $\mathsf{C}(A)$ is injective as an $R$-module.
\end{lemma}

\begin{proof}
  When $A$ is divisible, it is injective as an abelian group, and thus $\mathrm{Hom}_{\mathbb Z}(-, A) \cong \mathrm{Hom}_R(-, \mathsf{C}(A))$ is an exact functor.
  It follows that $\mathsf{C}(A)$ is injective as an $R$-module.
\end{proof}

Given an $R$-module $M$, choose an injective $R$-module $E$ containing $M$, and embed the underlying abelian group of $E/M$ into a divisible group $D$ by a map $j$, that is, $E\longrightarrow E/M\xlongrightarrow{j}D$.
By adjunction, such composition $E\longrightarrow D$ determines an $R$-linear map between injective modules with kernel $M$:
\begin{equation}\label{eq:copresentation}
  f:E\longrightarrow\mathsf{C}(D),\qquad
  f(e)(r)=j(re+M).
\end{equation}

\subsection{Proof of Theorem~A}

\begin{proof}[Proof of Theorem~A]
For $\alpha<\omega_1$, let $X_\alpha$ be the set of injections $\alpha\longrightarrow\omega$ with infinite complement.
Since every $\alpha<\omega_1$ is a countable ordinal, $X_\alpha$ is nonempty.
Restriction gives a surjection $\rho_{\alpha\beta}:X_\beta\longrightarrow X_\alpha$ for $\alpha\leq\beta<\omega_1$.
Note that $\varprojlim_{\alpha<\omega_1}X_\alpha=\emptyset$: a compatible family $(x_\alpha)_{\alpha<\omega_1}$ would give an injection $\bigcup_{\alpha<\omega_1}x_\alpha:\omega_1\longrightarrow\omega$, which is impossible.

\medskip

Put $V_\alpha=D^{(X_\alpha)}$, and $\varepsilon_\alpha:V_\alpha\longrightarrow D$ summing the coordinates. 
For $\alpha\leq\beta<\omega_1$, define the map
\[
u_{\alpha\beta}:V_\beta\longrightarrow V_\alpha,
\qquad
(v_t)_{t\in X_\beta}\longmapsto
\biggl(\sum_{\rho_{\alpha\beta}(t)=x}v_t\biggr)_{x\in X_\alpha}.
\]
Fix a section $s_{\alpha\beta}:X_\alpha\longrightarrow X_\beta$ of $\rho_{\alpha\beta}$.
It induces a homomorphism $t_{\alpha\beta}:V_\alpha\longrightarrow V_\beta$ given by
\[
\bigl(t_{\alpha\beta}(v)\bigr)_y=
\begin{cases}
v_x,&y=s_{\alpha\beta}(x)\text{ for some }x\in X_\alpha,\\
0,&y\notin s_{\alpha\beta}(X_\alpha).
\end{cases}
\]
Thus $u_{\alpha\beta}t_{\alpha\beta}=\mathrm{id}_{V_\alpha}$, so $u_{\alpha\beta}$ is a split epimorphism.
Moreover, $\varepsilon_\alpha u_{\alpha\beta}=\varepsilon_\beta$ and $\varepsilon_\beta t_{\alpha\beta}=\varepsilon_\alpha$.
Note that each $\varepsilon_\alpha$ is split epimorphic.

\medskip

We claim that $\varprojlim_{\alpha<\omega_1}V_\alpha=0$.
Let $(v_\alpha)_{\alpha<\omega_1}$ be a compatible family, so $u_{\alpha\beta}(v_\beta)=v_\alpha$ whenever $\alpha\leq\beta<\omega_1$.
Write $v_\alpha=((v_\alpha)_x)_{x\in X_\alpha}$ and put the finite set of nonzero coordinates of $v_\alpha$:
\[
T_\alpha=\operatorname{supp}(v_\alpha)
=\{x\in X_\alpha\mid (v_\alpha)_x\ne0\}.
\]
Compatibility gives $T_\alpha\subseteq\rho_{\alpha\beta}(T_\beta)$ and hence $|T_\alpha|\leq|T_\beta|$.
Note that $\sup _{\alpha < \omega_1}|T_\alpha|$ is finite; otherwise one can find an increasing indices $\alpha_n<\omega_1$ with $|T_{\alpha_n}|\geq n$, and then $|T_{\sup_{n<\omega}\alpha_n}|$ is infinite, a contradiction.
Hence $(|T_\alpha|)_{\alpha < \omega_1}$ is eventually constant.
On that tail, $\rho_{\alpha\beta}$ restricts to a bijection $T_\beta\longrightarrow T_\alpha$.
Hence $\varprojlim_{\alpha<\omega_1}X_\alpha$ contains a set of size $|T_\alpha|$.
Since $\varprojlim_{\alpha<\omega_1}X_\alpha=\emptyset$, all $T_\alpha$ are empty, which proves the claim.

\medskip

Now set the injective $R$-module $H_\alpha=\mathsf{C}(V_\alpha)$ and spilt epimorphism $q_\alpha=\mathsf{C}(\varepsilon_\alpha):H_\alpha\longrightarrow\mathsf{C}(D)$.
Since the right adjoint $\mathsf{C}$ preserves limits, $\varprojlim H_\alpha\cong \mathsf{C}(\varprojlim V_\alpha)=\mathsf{C}(0)=0$.
Using \eqref{eq:copresentation}, one has the pullbacks
\[
\xymatrix@C=4.5em@R=2em{
  I_\alpha \ar[r] \ar[d] & H_\alpha \ar[d]^-{q_\alpha} \\
  E \ar[r]^-{f} & \mathsf{C}(D)
}, \qquad I_\alpha=\{(e,h)\in E\oplus H_\alpha:f(e)=q_\alpha(h)\}.
\]
Since $\ker q_\alpha$ and $E$ are injective modules, their extension $I_\alpha$ is also an injective module.
The identity $\varepsilon_\alpha u_{\alpha\beta}=\varepsilon_\beta$ gives $q_\alpha\mathsf{C}(u_{\alpha\beta})=q_\beta$.
By universal property of pullbacks one has the $R$-linear maps
\[
p_{\alpha\beta}:I_\beta\longrightarrow I_\alpha,
\qquad (e,h)\longmapsto(e,\mathsf{C}(u_{\alpha\beta})(h)).
\]
Let $t_{\alpha\beta}:V_\alpha\longrightarrow V_\beta$ be the section of $u_{\alpha\beta}$ chosen above.
Since $\varepsilon_\beta t_{\alpha\beta}=\varepsilon_\alpha$, the map
\[
I_\alpha\longrightarrow I_\beta,
\qquad (e,h)\longmapsto(e,\mathsf{C}(t_{\alpha\beta})(h))
\]
is well defined and is a section of $p_{\alpha\beta}$.
In the inverse system $(I_\alpha)_{\alpha < \omega_1}$, a compatible family consists of the direct sum of some $e\in E$ and a compatible family $(h_\alpha)_{\alpha < \omega_1}$ in the $H_\alpha$, subject to $f(e)=q_\alpha(h_\alpha)$.
But $\varprojlim H_\alpha=0$, so every $h_\alpha$ is zero and the condition reduces to $f(e)=0$.
Thus $\varprojlim_{\alpha<\omega_1}I_\alpha\cong\ker f=M$.
\end{proof}

\begin{remark}\label{rem:countable-split}
The uncountable index in Theorem~A is necessary.
Indeed, one has the following results.
\begin{itemize}
  \item For countable ordinal $\lambda$ and a $\lambda$-indexed inverse limit of injective modules with split epimorphic connecting maps, say $(J_\alpha)_{\alpha < \lambda}$.
  Then $\varprojlim_{\alpha < \lambda} J_\alpha$ is again an injective module.
  \item For countable ordinal $\lambda$ and a $\lambda$-indexed inverse limit of injective modules with epimorphic connecting maps, say $(J'_\alpha)_{\alpha < \lambda}$.
  Then $\varprojlim_{\alpha < \lambda} J'_\alpha$ has injective dimension at most one.
\end{itemize}
\end{remark}

\subsection{Proof of Theorem~B}

\begin{proof}[Proof of Theorem~B]
Retain \eqref{eq:copresentation} and put the injective module $I=E\oplus\mathsf{C}(D^{(\omega)})$.
For $(e,u)\in I$, define
\begin{equation}\label{eq:insertion}
\sigma : I \longrightarrow I,\quad(e,u)\longmapsto(e,u'); \qquad u'(r)=\bigl(f(e)(r),u(r)_0,u(r)_1,\ldots\bigr) \quad(r\in R).  
\end{equation}
Then $\sigma$ is clearly an $R$-linear map with left inverse
\[
I \longrightarrow I,\quad(e,u)\longmapsto(e,u''); \qquad u''(r)=\bigl(u(r)_1, u(r)_2, u(r)_3,\ldots\bigr) \quad(r\in R). 
\]
Hence $\sigma$ is a split monomorphism.

\medskip

Iterating \eqref{eq:insertion} shows that $\bigcap_{n<\omega}\sigma^n(I)=M\oplus0$.
This proves the result.
\end{proof}

\section{The cotorsion criterion}\label{sec:cotorsion}

\subsection{The homological obstruction}

We first isolate a necessary condition on the ring.

\begin{lemma}\label{lem:dimension-obstruction}
For any ring $R$, every module in $\mathcal C_R$ has injective dimension at most one.

\medskip

Moreover, if $\operatorname{Cot}(R)\subseteq\mathcal C_R$, then $\operatorname{w.gl.dim}R\leq1$, i.e., every submodule of a flat module is flat.
\end{lemma}

\begin{proof}
For $L=\varprojlim(I_n,p_n)\in\mathcal L_R$, the difference map gives an exact sequence
\[
  0\longrightarrow L\longrightarrow\prod_n I_n
  \xlongrightarrow{\delta}\prod_n I_n\longrightarrow0,
  \qquad \delta((x_n))=(x_n-p_nx_{n+1}).
\]
Since each $p_n$ is surjective, a preimage under $\delta$ can be chosen recursively.
Products of injective modules are injective, so the displayed sequence is an injective resolution of $L$.
So $\operatorname{id}_R L\leq1$.
Note that any direct summand of $L$ has injective dimension at most one, so $\operatorname{id}_R C\leq1$ for every $C\in\mathcal C_R$.

\medskip

For a right $R$-module $N$, let $N^+=\operatorname{Hom}_{\mathbb Z}(N,\mathbb Q/\mathbb Z)$ be its character module.
Applying the tensor--Hom adjunction to a projective resolution, and using exactness of the character-dual functor, gives
\[
  \operatorname{Ext}^j_R(X,N^+)
  \cong\operatorname{Hom}_{\mathbb Z}
  \bigl(\operatorname{Tor}^R_j(N,X),\mathbb Q/\mathbb Z\bigr)
  \qquad(j\geq0).
\]
For flat $X$, the right side vanishes in degree one, so $N^+$ is cotorsion.
Under the assumed inclusion it has injective dimension at most one.
The formula in degree two, and the cogenerator property of $\mathbb Q/\mathbb Z$, now give $\operatorname{Tor}^R_2(N,X)=0$ for all $N$ and $X$.
Hence $\operatorname{w.gl.dim}R\leq1$.
\end{proof}

\subsection{Dedekind domains}

The next lemma supplies the Ext-orthogonality needed for a countably generated fraction field.

\begin{lemma}\label{lem:orthogonality}
Let $R$ be any ring, and let $P=\varinjlim_{i<\omega}P_i$ be a countably indexed filtered colimit of projective left $R$-modules with monomorphic transition maps.
Then $\operatorname{Ext}^1_R(P,L)=0$ for every $L\in\mathcal{L}_R$.
\end{lemma}

\begin{proof}
Write $L=\varprojlim_{n<\omega}(I_n,p_n)$ with $p_n:I_{n+1}\longrightarrow I_n$ surjective and every $I_n$ an injective module.
Take an arbitrary extension $0\longrightarrow L\longrightarrow X\xlongrightarrow{q}P\longrightarrow0$.
Regard each $P_n$ as a submodule of $P$, and put $X_n=q^{-1}(P_n)$.
Then $X=\bigcup_nX_n$.
Since $P_n$ is projective, each $0\longrightarrow L\longrightarrow X_n\longrightarrow P_n\longrightarrow0$ splits.
Let $\pi_n:L\longrightarrow I_n$ be the limit projections.

\medskip

We then construct maps $a_n:X_n\longrightarrow I_n$ extending $\pi_n$ so that the squares
\[
\xymatrix@C=4em@R=2em{
X_n\ar@{>->}[r]\ar[d]_{a_n}\ar@{-->}[dr]^{b_n}
&X_{n+1}\ar[d]^{a_{n+1}}\\
I_n&I_{n+1}\ar@{->>}[l]_-{p_n}
}
\]
commute. 
Here $a_0$ follows from that $I_0$ is an injective module.
Given $a_n$, then $a_{n+1} : X_{n+1} \longrightarrow I_{n+1}$ is defined as follows. 
Fix the direct sums $X_{n}\cong L\oplus P_{n}$.
On the first summand use $\pi_{n+1}$, and on the second summand lift the second component of $a_n$ through $p_n$, since $P_n$ is a projective module and $p_n$ is surjective.
This gives $b_n:X_n\longrightarrow I_{n+1}$ with $p_nb_n=a_n$ and $b_n|_L=\pi_{n+1}$.
Since $I_{n+1}$ is an injective module, $b_n$ extends to $a_{n+1}:X_{n+1}\longrightarrow I_{n+1}$.
This gives the construction.

\medskip

For each fixed $i$, the restrictions $a_n|_{X_i}$ for $n\geq i$ form a compatible family and hence define a map $s_i:X_i\longrightarrow L$.
These maps agree on $X_i\subseteq X_{i+1}$ and restrict to the identity on $L$, since $a_n|_L=\pi_n$.
Thus they induce a retraction $s:X=\bigcup_iX_i\longrightarrow L$.
Thus the extension splits and $\operatorname{Ext}^1_R(P,L)=0$.
\end{proof}

\begin{lemma}\label{lem:cotorsion-field}
Let $R$ be a Dedekind domain with fraction field $K=\operatorname{Frac}(R)$.
For every $R$-module $C$
\begin{equation}\label{eq:cotorsion-field}
  C\in\operatorname{Cot}(R)\quad\Longleftrightarrow\quad\operatorname{Ext}^1_R(K,C)=0.
\end{equation}
\end{lemma}

\begin{proof}
  ($\implies$) follows from the definition.

  \medskip

  For the converse ($\impliedby$), take any flat $R$-module $F$.
  Then the canonical map $F \cong F \otimes _RR \longrightarrow K\otimes_R F$ is monic.
  Since $K\otimes_R F$ is a $K$-vector space and hence isomorphic to $K^{(J)}$ as an $R$-module.
  Since the Dedekind domain is hereditary, one has the associated Ext sequence contains
\[
  \operatorname{Ext}^1_R(K^{(J)},C)\longrightarrow\operatorname{Ext}^1_R(F,C)
  \longrightarrow\operatorname{Ext}^2_R(K^{(J)}/F,C)=0.
\]
By assumption $\operatorname{Ext}^1_R(K^{(J)},C)\cong\prod_J\operatorname{Ext}^1_R(K,C)=0$, one has taht $\operatorname{Ext}^1_R(F,C)=0$.
Since $F$ arbitrary chosen, $C$ is cotorsion.
\end{proof}

\begin{lemma}\label{lem:fraction-countability}
  Let $R$ be a Dedekind domain with fraction field $K$.
  Then $|\operatorname{Spec}R|\leq\aleph_0$ if and only if $K$ is at most countably generated as an $R$-module.
\end{lemma}

\begin{proof}
The assertion is immediate if $R$ is a field.

\medskip

($\implies$) Suppose that $|\operatorname{Spec}R|\leq\aleph_0$.
For each nonzero prime ideal $\mathfrak p$, choose $0\ne c_{\mathfrak p}\in\mathfrak p$, and let $U$ be the multiplicative set generated by these elements.
The set $U$ is at most countable.
Given $0\ne b\in R$, factor its principal ideal as $bR=\prod_{j=1}^m\mathfrak p_j^{e_j}$.
Then $u=\prod_{j=1}^m c_{\mathfrak p_j}^{e_j}$ belongs to $bR$, so $\frac1b=\frac{u/b}{u}\in U^{-1}R$.
Hence $K=U^{-1}R=\sum_{u\in U}R(1/u)$ is at most countably generated over $R$.

\medskip

($\impliedby$)
Conversely, suppose first that $K=\sum_{n<\omega}R(a_n/b_n)$, where each $b_n\ne0$.
Fix a nonzero prime ideal $\mathfrak p$.
If no $b_n$ belonged to $\mathfrak p$, then every $a_n/b_n$ would belong to $R_{\mathfrak p}$, and hence $K\subseteq R_{\mathfrak p}$.
This is impossible: for $0\ne r\in\mathfrak p$, one has $1/r\notin R_{\mathfrak p}$.
Thus $\mathfrak p$ contains some $b_n$.
Each nonzero $b_n$ lies in only finitely many prime ideals, by the factorization of $b_nR$ into nonzero prime ideals.
Hence $|\operatorname{Spec}R|\leq\aleph_0$.
\end{proof}

\begin{proposition}\label{prop:stabilization}
Let $R$ be a Dedekind domain and $C$ a cotorsion $R$-module.
Then $D\oplus C\in\mathcal{L}_R$ for some injective $R$-module $D$.
Consequently $\operatorname{Cot}(R)\subseteq\mathcal C_R$.
\end{proposition}

\begin{proof}
  Fix such $C$ and choose an epimorphism $s:F\longrightarrow C$ with $F$ free.
  Put $P=F^\omega$, $U=F^{(\omega)}$ and $H=P/U$.
The diagonal map $j:F\longrightarrow H$ is injective.
Note that the cokernel $H/j(F)$ is torsion-free.
Indeed, if $0\ne r\in R$ and $r[(x_n)]=j(f)$, then $rx_n=f$ for all sufficiently large $n$.
Multiplication by $r$ is injective on $F$, so the $x_n$ are eventually equal and $[(x_n)]\in j(F)$.
Over a Dedekind domain, torsion-free modules are flat.
Thus $\operatorname{Ext}^1_R(H/j(F),C)=0$, and the Hom--Ext sequence extends $s$ to a map $\varphi:H\longrightarrow C$ satisfying $\varphi j=s$.
In particular, $\varphi$ is surjective.

\medskip

Let $V$ be the kernel of $P\longrightarrow H\xlongrightarrow{\varphi}C$, and set
\[
  V_n=\{v\in V:v_0=\cdots=v_{n-1}=0\}\qquad(n<\omega).
\]
We have $U\subseteq V$, $P/V\cong C$, and $V_0=V$.
Projection to the first $n$ coordinates identifies $V/V_n$ with $F^n$; every finite prefix has a representative in $U\subseteq V$.
Consequently
\begin{equation}\label{eq:completion}
  \bigcap_n V_n=0,\qquad
  \varprojlim_n V/V_n\cong P,
\end{equation}
where the natural map from $V$ to this limit agrees with its original inclusion in $P$.

\medskip

Put $K=\operatorname{Frac}(R)$ and embed $V$ in the $K$-vector space $D=K\otimes_R V$.
Over a Dedekind domain, divisible modules are injective, so $D$ and every $D/V_n$ are injective.
The quotient maps form an tower.
Put $L=\varprojlim_n D/V_n$.
By \eqref{eq:completion}, the natural map $D\longrightarrow L$ is injective, and the kernel of the structural projection $L\longrightarrow D/V$ is $\varprojlim_n V/V_n\cong P$.
The composite $D\longrightarrow L\longrightarrow D/V$ is the quotient map, so $L=D+P$.
Moreover, $D\cap P=V$: both embeddings agree on $V$, and any element of $D\cap P$ maps to zero in $D/V$.
Therefore $L/D\cong P/V\cong C$.
Since $D$ is injective, $0\longrightarrow D\longrightarrow L\longrightarrow C\longrightarrow0$ splits, and $L\cong D\oplus C$.
\end{proof}

The converse of Proposition~\ref{prop:stabilization} fails.

\begin{proposition}\label{prop:uncountable-spectrum}
Let $R$ be a Dedekind domain with uncountable spectrum and fraction field $K$, and put $B=R^\omega+K^{(\omega)}\subseteq K^\omega$.
Then $B\in\mathcal{L}_R$, but $\operatorname{Ext}^1_R(K,B)\ne0$.
Consequently $\operatorname{Cot}(R)\subsetneq\mathcal C_R$.
\end{proposition}

\begin{proof}
The modules $K$ and $K/R$ are injective, and direct sums of injectives are injective over the Noetherian ring $R$.
Applying \cite[Corollary~11(c)]{Bergman} to $0\longrightarrow R\longrightarrow K\longrightarrow K/R\longrightarrow0$ gives $B\in\mathcal L_R$.

\medskip

Let $\Omega=\operatorname{spec} R \setminus \{0\}$ be the uncountable set of maximal ideals of $R$.
For $\mathfrak p\in\Omega$, put
\[
  T_{\mathfrak p}=\bigcup_{m\geq0}\mathfrak p^{-m}\subseteq K,
  \qquad U_{\mathfrak p}=T_{\mathfrak p}/R.
\]
The primary decomposition of the torsion module $K/R$ gives $K/R=\bigoplus_{\mathfrak p\in\Omega}U_{\mathfrak p}$.
Apply $\mathrm{Hom}_R(-, B)$ to $0\longrightarrow R\longrightarrow T_{\mathfrak p}\longrightarrow U_{\mathfrak p}\longrightarrow0$ yields a boundary map $\delta_{\mathfrak p}:B\longrightarrow\operatorname{Ext}^1_R(U_{\mathfrak p},B)$.
We claim that
\begin{equation}\label{eq:eventual-prime}
  \delta_{\mathfrak p}(b)=0
  \quad\Longrightarrow\quad
  b_n\in\mathfrak p\text{ for all sufficiently large }n.
\end{equation}
Indeed, such vanishing extends $1\longmapsto b$ to a map $h:T_{\mathfrak p}\longrightarrow B$.
Choose $z\in\mathfrak p^{-1}\setminus R$.
Since $\mathfrak p$ is invertible, one has $z$ with $(R:z)=\mathfrak p$.
For $x\in T_{\mathfrak p}$, multiplication by a nonzero denominator shows that $h(x)_n=xb_n$ in $K$.
Since $b,h(z)\in B$, both $b_n$ and $zb_n=h(z)_n$ lie in $R$ for all sufficiently large $n$.
Thus $b_n\in(R:z)= \{a \in R \mid az \in R\} =\mathfrak p$, proving \eqref{eq:eventual-prime}.
In particular, $\delta_{\mathfrak p}(\mathbf1)\ne0$ for $\mathbf1=(1,1,\ldots)$.

\medskip

Applying $\operatorname{Hom}_R(-,B)$ to $0\longrightarrow R\longrightarrow K\longrightarrow K/R\longrightarrow0$ gives the exact sequence
\begin{equation}\label{eq:primary-ext}
  B\xlongrightarrow{\delta}
  \prod_{\mathfrak p\in\Omega}\operatorname{Ext}^1_R(U_{\mathfrak p},B)
  \longrightarrow\operatorname{Ext}^1_R(K,B)\longrightarrow0,
\end{equation}
where the $\mathfrak p$-coordinate of $\delta$ is $\delta_{\mathfrak p}$: pulling back the sequence along $U_{\mathfrak p}\longrightarrow K/R$ gives $0\longrightarrow R\longrightarrow T_{\mathfrak p}\longrightarrow U_{\mathfrak p}\longrightarrow0$.
Fix $\mathfrak p_0\in\Omega$ and define an element $\eta=(\eta_{\mathfrak p})_{\mathfrak p\in\Omega}$ of the product in \eqref{eq:primary-ext} by
\[
  \eta_{\mathfrak p}=
  \begin{cases}
    \delta_{\mathfrak p_0}(\mathbf1),&\mathfrak p=\mathfrak p_0,\\
    0,&\mathfrak p\ne\mathfrak p_0.
  \end{cases}
\]
If $\eta=\delta(b)$, then for each $\mathfrak p\ne\mathfrak p_0$, \eqref{eq:eventual-prime} gives an integer $N_{\mathfrak p}$ such that $b_n\in\mathfrak p$ for all $n\geq N_{\mathfrak p}$.
The countably many sets $\{\mathfrak p\ne\mathfrak p_0:N_{\mathfrak p}=N\}$, indexed by $N<\omega$, cover the uncountable set $\Omega\setminus\{\mathfrak p_0\}$.
Hence one of these sets is infinite, indexed by $N$.
Since $b_n\in R$ for sufficiently large $n$ and a nonzero element of $R$ lies in only finitely many maximal ideals, $b_n=0$ for all sufficiently large $n$.
But then $b\in K^{(\omega)}$, and $x\longmapsto(xb_n)_n$ extends $1\longmapsto b$ to every $T_{\mathfrak p}$.
Thus $\delta(b)=0$, contradicting $\eta\ne0$.
The class of $\eta$ in \eqref{eq:primary-ext} proves $\operatorname{Ext}^1_R(K,B)\ne0$.
Since $K$ is flat, $B\notin\operatorname{Cot}(R)$.
As $B \in \mathcal{L}_R \subseteq \mathcal{C}_R$, one has $\operatorname{Cot}(R)\subsetneq\mathcal C_R$.
\end{proof}

\begin{proposition}\label{prop:dedekind-criterion}
Let $R$ be a Dedekind domain with fraction field $K$.
Then the following are equivalent:
\[
  \mathcal C_R=\operatorname{Cot}(R)
  \quad\Longleftrightarrow\quad
  |\operatorname{Spec}R|\leq\aleph_0
  \quad\Longleftrightarrow\quad
  K\text{ is at most countably generated over }R.
\]
Moreover, $\operatorname{Cot}(R)\subseteq\mathcal C_R$ for every Dedekind domain $R$.
\end{proposition}

\begin{proof}
Proposition~\ref{prop:stabilization} gives
$\operatorname{Cot}(R)\subseteq\mathcal C_R$ for every Dedekind domain, 
and Lemma~\ref{lem:fraction-countability} gives the second ($\Longleftrightarrow$).
It remains to prove the first ($\Longleftrightarrow$).

\medskip

Suppose $|\operatorname{Spec}R|\leq\aleph_0$.
Then $K$ is a countable increasing union of finitely generated torsion-free submodules $P_n$, each of which is projective.
Lemma~\ref{lem:orthogonality} gives $\operatorname{Ext}^1_R(K,L)=0$ for every $L\in\mathcal L_R$.
The same vanishing holds for direct summands of such $L$.
Lemma~\ref{lem:cotorsion-field} therefore gives $\mathcal C_R\subseteq\operatorname{Cot}(R)$, and the first paragraph gives equality.

\medskip

Conversely, suppose $\mathcal C_R=\operatorname{Cot}(R)$.
If the spectrum were uncountable, Proposition~\ref{prop:uncountable-spectrum} would give a module in $\mathcal L_R\subseteq\mathcal C_R$ that is not cotorsion, a contradiction.
Hence $|\operatorname{Spec}R|\leq\aleph_0$.
\end{proof}

\subsection{Proof of Theorem~C}

\begin{proof}[Proof of Theorem~C]
Suppose first that $\mathcal C_R=\operatorname{Cot}(R)$.
Lemma~\ref{lem:dimension-obstruction} gives $\operatorname{w.gl.dim}R\leq1$.
For every ideal $J$, dimension shifting in $0\longrightarrow J\longrightarrow R\longrightarrow R/J\longrightarrow0$ shows that $J$ is flat.
Since $R$ is Noetherian, $J$ is a finitely presented flat module and hence projective.
Thus every ideal is projective, so $R$ is hereditary and $\operatorname{gldim}R\leq1$.
Recall that a commutative Noetherian hereditary ring is a finite product $R\cong R_1\times\cdots\times R_t$ of Dedekind domains, with field factors allowed; see \cite[Section~16.5, p.~399]{GT}.
Under this decomposition, injectivity, flatness, inverse limits and direct summands are computed componentwise, as is Ext.
Hence $\mathcal C_{R_i}=\operatorname{Cot}(R_i)$ for every $i$.
Proposition~\ref{prop:dedekind-criterion} makes every $\operatorname{Spec}R_i$ countable, proving necessity.

\medskip

Conversely, if $\operatorname{gldim}R\leq1$ and $\operatorname{Spec}R$ is countable, the same decomposition has Dedekind factors with countable spectra.
Proposition~\ref{prop:dedekind-criterion}, applied componentwise, gives $\mathcal C_R=\operatorname{Cot}(R)$.
\end{proof}

\end{document}